\documentclass[10pt,reqno]{amsart}

\pdftrailerid{}
\usepackage[T1]{fontenc}
\usepackage{lmodern}
\usepackage[margin=1.08in]{geometry}
\usepackage{microtype}
\usepackage{amsmath,amssymb,mathtools}
\usepackage{enumitem}
\usepackage[hidelinks]{hyperref}
\hypersetup{
  pdftitle={The Gross Property for Implicit Functions},
  pdfauthor={Sina Nadi},
  pdfsubject={The Gross Property for Implicit Functions},
  pdfcreator={},
  pdfproducer={},
  pdfkeywords={},
}

\setlist[enumerate]{label=(\roman*),leftmargin=2.15em,itemsep=0.15em,topsep=0.35em}
\setlist[itemize]{leftmargin=1.8em,itemsep=0.15em,topsep=0.35em}
\numberwithin{equation}{section}

\newtheorem{theorem}{Theorem}[section]
\newtheorem{lemma}[theorem]{Lemma}
\newtheorem{proposition}[theorem]{Proposition}

\newcommand{\C}{\mathbb C}
\newcommand{\D}{\mathbb D}
\newcommand{\T}{\mathbb T}
\newcommand{\Z}{\mathbb Z}
\newcommand{\eps}{\varepsilon}
\newcommand{\dist}{\operatorname{dist}}
\newcommand{\diam}{\operatorname{diam}}
\newcommand{\Real}[1]{\mathit{Re}\{#1\}}
\newcommand{\calM}{\mathcal M}
\newcommand{\calS}{\mathcal S}
\newcommand{\lemref}[1]{\hyperref[#1]{Lemma~\ref*{#1}}}
\newcommand{\propref}[1]{\hyperref[#1]{Proposition~\ref*{#1}}}
\newcommand{\thmref}[1]{\hyperref[#1]{Theorem~\ref*{#1}}}
\newcommand{\secref}[1]{\hyperref[#1]{Section~\ref*{#1}}}

\makeatletter
\def\@secnumfont{\bfseries}
\def\section{\@startsection{section}{1}%
  \z@{.7\linespacing\@plus\linespacing}{.5\linespacing}%
  {\normalfont\bfseries\centering}}
\makeatother

\title[The Gross Property for Implicit Functions]
{The Gross Property for Implicit Functions}
\author[Sina Nadi]{Sina Nadi}
\date{}

\begin{document}

\begin{abstract}
In 1918, Gross proved that a regular local inverse of a meromorphic
function in the plane can be continued analytically along every ray from
its centre except for directions in a set of Lebesgue measure zero.
Eremenko asked whether the same conclusion holds for an implicit function
defined by an entire relation in two variables. We prove that there exist
an entire function $F$ of two variables and a regular implicit germ
$\varphi$, defined by $F(z,\varphi(z))=0$, whose analytic continuation
fails on every ray from its centre. In fact, along each ray, the modulus
of the continuation tends to infinity as the singular point is approached.
\end{abstract}

\maketitle

\section{Introduction}\label{sec:introduction}

For $z\in\C$ and $r>0$, we write
$\Delta(z,r)=\{w\in\C:|w-z|<r\}$, $\D=\Delta(0,1)$, and
$\T=\partial\D$. The space of holomorphic functions on a complex manifold
$X$ is denoted by $\mathcal O(X)$. We also put $\C^*=\C\setminus\{0\}$.
Let $F\in\mathcal O(\C^2)$ and suppose that
\[
 F(z_0,w_0)=0,
 \qquad F_w(z_0,w_0)\ne0.
\]
The implicit function theorem gives a unique holomorphic germ $\varphi$
at $z_0$ that satisfies
\[
 F(z,\varphi(z))=0,
 \qquad \varphi(z_0)=w_0.
\]
We call $\varphi$ a regular implicit germ of $F$. We use the same symbol
for the germ and its analytic continuation along a fixed path.
We say that $\varphi$ has the \emph{Gross property} if it admits analytic
continuation along
\[
 \{z_0+t e^{i\theta}:t\geq0\}
\]
for almost every $\theta\in[0,2\pi)$ with respect to Lebesgue measure.

In 1918, Gross proved that a regular local inverse of a meromorphic
function in the plane can be continued analytically along almost every
ray from its centre \cite{Gross1918}. In particular, if $f$ is entire and
$F(z,w)=z-f(w)$, then the implicit germ $\varphi$ has the Gross property.

In 1936, Sto\"ilow proved that implicit functions defined by entire
relations have the Iversen property \cite{Stoilov1936}, which means that
for every continuous path $\gamma\colon[0,1]\to\C$ with $\gamma(0)=z_0$
and every $\eps>0$, there
is a continuous path $\gamma_1\colon[0,1]\to\C$ with $\gamma_1(0)=z_0$
such that
\[
 |\gamma(t)-\gamma_1(t)|\leq\eps
 \qquad(0\leq t\leq1),
\]
and $\varphi$ admits analytic continuation along $\gamma_1$
\cite{Eremenko2015}. This extends Iversen's theorem for inverse functions
of meromorphic functions in the plane \cite[pp.~24--25]{Iversen1914}.

Eremenko posed the following problem \cite{Eremenko2015}.

\medskip
\noindent\textbf{Eremenko's problem.}
Does every regular implicit germ defined by an entire relation in two
variables admit analytic continuation along almost every ray from its
centre?
\medskip

We prove that there is a regular implicit germ whose analytic
continuation fails on every ray from its centre.

We call an entire function reduced if its divisor of zeros has
multiplicity one along every irreducible component.

\begin{theorem}\label{thm:main}
There exist a reduced entire function $F\in\mathcal O(\C^2)$ and a point
$(z_0,w_0)\in\C^2$ such that
\[
 F(z_0,w_0)=0,
 \qquad F_w(z_0,w_0)\ne0.
\]
Let $\varphi$ be the corresponding implicit germ at $z_0$ with
$\varphi(z_0)=w_0$. For every $\theta\in[0,2\pi)$, there is
$0<R_\theta<\infty$ such that $\varphi$ admits analytic continuation along
$z_0+t e^{i\theta}$ for $0\leq t<R_\theta$, and
\[
 |\varphi(z_0+t e^{i\theta})|\longrightarrow+\infty
 \qquad(t\to R_\theta^-).
\]
\end{theorem}

Thus the exceptional set is the whole circle of directions.

In 1988, Stephenson constructed an inner function with a local inverse
that cannot be continued analytically along any ray from its centre
\cite{Stephenson1988}. We use a similar construction in
\propref{prop:blocking}.

In 1975, Alexander proved that a universal covering of $\C\setminus E$,
where $E$ is closed, has logarithmic capacity zero, and contains at least
two points, can be the first coordinate of a proper holomorphic map from
$\D$ to $\C^2$ \cite[p.~327]{Alexander1975}. Here we construct a proper
holomorphic immersion whose first coordinate has a local inverse that
fails to continue through a finite point on every ray from its centre,
and the second coordinate tends to infinity in modulus along the
corresponding lifts.

In \secref{sec:blocking}, we construct the Riemann surface and study its
boundary. In \secref{sec:barrier}, we use a peak function for the disc
algebra to construct a holomorphic function.
This function is used in \secref{sec:completion} to obtain a proper
holomorphic immersion into $\C^2$. We prove \thmref{thm:main} in
\secref{sec:main-proof}.

\section{Construction of a Riemann domain}\label{sec:blocking}

We denote normalized arclength measure on $\T$ by $m$ and write
$\omega(z,E,\Omega)$ for harmonic measure in a plane domain $\Omega$.
See \cite{GarnettMarshall2005}.
The notation $A\Subset B$ means that the closure of $A$ is a compact
subset of the interior of $B$.

A \emph{Riemann domain over $\D$} will mean a connected Riemann surface
$X$ together with a locally biholomorphic map $\pi_X\colon X\to\D$.
For the covering-space terminology used below, see
\cite[Chapter~1]{Forster1981}. If $X$ is simply connected and
$\pi_X$ is nonconstant, then $X$ is conformally equivalent to $\D$.
Indeed, $X$ is noncompact, and a uniformization by $\C$ would make
$\pi_X$ a bounded nonconstant entire function. When a uniformization
$\Psi\colon\D\to X$ is fixed, boundary sets of $X$ will be represented on
$\T$ through $\Psi$. We use this representation in all harmonic measure
statements below.

We begin with a harmonic measure estimate.

\begin{lemma}\label{lem:small-gates}
Let $E_1,E_2\subset\D$ be disjoint nonempty compact connected sets such
that $\D\setminus(E_1\cup E_2)$ is connected, and let $M$ be a nonempty
compact subset of this domain. Suppose that
\[
 \diam E_j\leq\delta,
 \qquad \dist(M,E_j)\geq d>0,
 \qquad j=1,2.
\]
If $0<\delta<1$, then
\begin{equation}\label{eq:small-gate-bound}
 \sup_{w\in M}
 \omega\bigl(w,E_1\cup E_2,\D\setminus(E_1\cup E_2)\bigr)
 \leq
 \frac{2\log(3/d)}{\log(3/\delta)}.
\end{equation}
In particular, the left-hand side tends to zero as $\delta\to0$,
uniformly over all configurations that satisfy the displayed bounds.
\end{lemma}

\begin{proof}
Choose $\xi_j\in E_j$. Since
$E_j\subseteq\overline{\Delta(\xi_j,\delta)}$, the function
\[
 b_j(z)=\frac{\log(3/|z-\xi_j|)}{\log(3/\delta)}
\]
is positive and harmonic on $\D\setminus\{\xi_j\}$. It is at least one
on $E_j\setminus\{\xi_j\}$, tends to $+\infty$ as $z\to\xi_j$, and is
nonnegative on $\T$. If
\[
 h(z)=\omega\bigl(z,E_1\cup E_2,
                  \D\setminus(E_1\cup E_2)\bigr),
\]
then the maximum principle gives $h\leq b_1+b_2$. For $w\in M$ we
have $|w-\xi_j|\geq d$, and \eqref{eq:small-gate-bound} follows.
\end{proof}

We next describe the cutting operation used in the construction. The
geometric terminology follows Stephenson \cite{Stephenson1988}.
By a bank of a cut, we mean one of the two boundary copies of the cut.

\begin{lemma}\label{lem:wing}
Let $\sigma\subset\D$ be a simple analytic arc with distinct endpoints
$b,c\in\D$, and let
\[
 q\colon X\longrightarrow\D\setminus\{b,c\}
\]
be a universal covering. If $\widetilde\sigma$ is a lift of
$\sigma\setminus\{b,c\}$, then $\widetilde\sigma$ is a properly
embedded copy of $\mathbb R$ in $X$. Cutting $X$ along
$\widetilde\sigma$ gives two simply connected surfaces with boundary,
and the covering projection supplies holomorphic coordinates across either
bank of the cut.

Let $J$ be an open arc of $\T$, and let $C_J\subset\D$ be a simply
connected Jordan domain whose boundary contains $J$ and whose closure in
$\overline\D$ avoids $b$ and $c$. Every component of $q^{-1}(C_J)$ is
mapped biholomorphically onto $C_J$ and has an analytic boundary arc
that projects onto $J$.
\end{lemma}

\begin{proof}
The arc $\sigma\setminus\{b,c\}$ is a proper embedding of
$\mathbb R$ in $\D\setminus\{b,c\}$, and each lift is proper and
embedded. The universal covering surface $X$ is homeomorphic to the
plane. The Schoenflies theorem, applied after one-point
compactification, shows that $X\setminus\widetilde\sigma$ has two
components, each homeomorphic to a half-plane.
Therefore the first assertion follows. Near the cut, the covering
projection itself is a local coordinate.

Since $C_J$ is simply connected and avoids the punctures, the restriction
of $q$ to each component of $q^{-1}(C_J)$ is a one-sheeted covering,
hence a biholomorphism. Adjoining the lift of $J$ gives the stated
analytic boundary arc.
\end{proof}

We call a component of $q^{-1}(C_J)$, with its boundary arc over $J$
adjoined, a \emph{lifted collar over $J$}.

By a crosscut of a domain, we mean a simple open arc in the domain whose
closure is a closed arc with two distinct endpoints on its boundary.

A corridor is a Jordan domain $Q\Subset\D$ whose boundary consists of
four analytic arcs. We denote the relative interiors of these arcs, in
cyclic order, by
\[
 \tau,\ \sigma_1,\ \xi,\ \sigma_0.
\]
Following Stephenson, we call $\tau$ the \emph{threshold}, $\xi$ the
\emph{exit}, and $\sigma_0,\sigma_1$ the lateral sides. Let
$\Gamma\subset\D$ be a Jordan curve disjoint from
$\overline\tau\cup\overline\xi$ such that the threshold and exit lie in
different components of $\D\setminus\Gamma$. In the application,
$\Gamma$ is a level curve of the prescribed local coordinate, and
$Q\cap\Gamma$ is the middle crosscut of the corridor.

\begin{lemma}\label{lem:assembly}
Let $Q$ and $\Gamma$ be defined as above, and put
\[
 E=\overline\tau\cup\overline\xi.
\]
There is a simply connected Riemann domain
$\pi_Q\colon A_Q\to\D$ with the following properties.

\begin{enumerate}
\item The domain $Q$ is contained in $A_Q$, and $\pi_Q$ is the
identity on $Q$. The arcs $\tau$ and $\xi$ can be adjoined as disjoint
analytic boundary arcs, with their endpoints omitted.

\item The closed set
\begin{equation}\label{eq:assembly-middle-set}
 \mathcal M_Q=\pi_Q^{-1}(\Gamma)
\end{equation}
separates the threshold from the exit. There are collars of the threshold
and exit such that every path in $A_Q$ joining these collars meets
$\mathcal M_Q$.

\item Fix a conformal map $\Psi\colon\D\to A_Q$. The threshold and exit
correspond to open arcs $T_\tau,T_\xi\subset\T$. There is a measurable
set
\[
 B_Q\subseteq\T\setminus(T_\tau\cup T_\xi)
\]
such that
\[
 m\bigl((\T\setminus(T_\tau\cup T_\xi))\setminus B_Q\bigr)=0,
\]
and, for every $\zeta\in B_Q$, the curve $\Psi(r\zeta)$ eventually
lies in a lifted collar over an arc of $\T$. If $x\in\mathcal M_Q$ and
$w=\pi_Q(x)$, then
\begin{equation}\label{eq:assembly-comparison}
 \omega\bigl(\Psi^{-1}(x),B_Q,\D\bigr)
 \geq \omega(w,\T,\D\setminus E).
\end{equation}
\end{enumerate}
\end{lemma}

\begin{proof}
For each lateral side $\sigma_j$, apply \lemref{lem:wing} to
$\overline{\sigma_j}$ and retain the component which projects locally to
the exterior of $Q$. Attach its cut bank to $\sigma_j$. We call the
attached surfaces wings and the resulting surface an assembly,
following \cite{Stephenson1988}.

The identifications are made in analytic collars, using the projection to
$\D$ as local coordinate. The quotient is therefore a Hausdorff Riemann
surface and the induced projection $\pi_Q$ is locally biholomorphic. The
corridor and the two wings are simply connected, and their incidence graph
is a tree. Van Kampen's theorem shows that $A_Q$ is simply connected.
If a path joins the threshold side to the exit side, its projection joins
the two components of $\D\setminus\Gamma$ containing those sides. The
projection must meet $\Gamma$, and the path therefore meets
$\mathcal M_Q$. Notice that $\mathcal M_Q$ is the full inverse image of
$\Gamma$ in the assembly. It includes any components in the wings as well
as the part in the corridor.

For (iii), let $\Omega=\D\setminus E$ and set
\[
 u(z)=\omega(z,\T,\Omega).
\]
Since $E$ consists of two disjoint analytic arcs, every boundary point
of $\Omega$ is regular for the Dirichlet problem. Hence $u$ extends
continuously to $\overline\Omega$, with boundary values zero on $E$
and one on $\T$. Extending $u$ by zero across $E$ gives a bounded
subharmonic function $\widehat u$ on $\D$.

In a boundary coordinate that straightens an adjoined analytic boundary
arc, the inverse of the uniformization is univalent. Its modulus tends
to one at each point of the arc. Otherwise, a sequence approaching that
point would have a subsequence that converges to an interior point of
the surface, contrary to the Hausdorff property after adjoining the arc.
The reflection principle therefore extends the inverse
conformally across the arc.

Put $f=\pi_Q\circ\Psi$. Applying this argument to the threshold and exit
gives disjoint open arcs $T_\tau,T_\xi\subset\T$ corresponding to them.
The function $\widehat u\circ f$ is bounded and subharmonic in $\D$.
Its boundary limit is zero on $T_\tau\cup T_\xi$, and its upper boundary limit is at
most one elsewhere. The maximum principle for bounded subharmonic
functions therefore gives
\begin{equation}\label{eq:assembly-perron}
 \omega\bigl(\Psi^{-1}(x),
       \T\setminus(T_\tau\cup T_\xi),\D\bigr)
 \geq u(w).
\end{equation}

It remains to determine the boundary values of $f$ outside
$T_\tau\cup T_\xi$. Let $S_Q$ be the set of projections of the four
omitted corners. Suppose that $\zeta\notin T_\tau\cup T_\xi$ and
$f(r\zeta)\to\lambda\in\D\setminus S_Q$. Choose concentric discs
$B'\Subset B\Subset\D\setminus S_Q$ about $\lambda$ so small that $B$
meets at most one side of the corridor, in a single analytic arc.
The tail of $\Psi(r\zeta)$ lies in one component $Y$ of
$\pi_Q^{-1}(B)$.

Over $B$, the covering pieces are one-sheeted. If $B$ meets a lateral
side, gluing the corresponding bank to the corridor completes the local
coordinate across that side. Thus either $Y$ is mapped biholomorphically
onto $B$, or it is a boundary chart on one side of a threshold or
exit arc. In the first case, $Y\cap\pi_Q^{-1}(\overline{B'})$ is compact
in $A_Q$, which is impossible because $\Psi(r\zeta)$ leaves every
compact subset as $r\to1$. In the second case, the local boundary
correspondence gives $\zeta\in T_\tau\cup T_\xi$, again a contradiction.
Consequently, every interior radial limit of $f$ outside these two arcs
belongs to $S_Q$.

By Fatou's theorem, the bounded analytic function $f$ has angular
limits almost everywhere, and these agree with its radial limits $f^*$.
For a fixed $\lambda\in S_Q$, Privalov's boundary uniqueness theorem
\cite[Corollary~6.14, p.~140]{Pommerenke1992}, applied to $f-\lambda$,
gives
\[
 m\{\zeta\in\T:f^*(\zeta)=\lambda\}=0.
\]
It follows that, for almost every
$\zeta\in\T\setminus(T_\tau\cup T_\xi)$, the projection of
$\Psi(r\zeta)$ tends to $\T$. The corridor and the two
attaching banks project into a compact subset of $\D$. Such a curve must
therefore eventually lie in one wing and then in a lifted collar as in
\lemref{lem:wing}. Choose the boundary arcs of the collars from a
countable basis of $\T$, with the closures of the collars disjoint from
$\overline Q$. The local boundary correspondence identifies each lifted
boundary arc with an open arc in $\T$. Since each covering has
countably many components over a fixed collar, these arcs form a
countable family. Let $B_Q$ be their union. Then $B_Q$ is measurable
and has full measure in $\T\setminus(T_\tau\cup T_\xi)$. Equation
\eqref{eq:assembly-comparison} now follows from
\eqref{eq:assembly-perron}.
\end{proof}

We now carry out a construction similar to Stephenson's
\cite{Stephenson1988} in a prescribed local holomorphic coordinate.

\begin{proposition}\label{prop:blocking}
Let $p\in\mathcal O(\D)$ be nonconstant, and let $a\in\D$ satisfy
$p'(a)\ne0$. There exist a locally univalent inner function
$I\colon\D\to\D$, $R_*>0$, a compact set $K\subset\D$, and an
open set $U\subset\T$ such that
\[
 I(0)=a,
 \qquad I'(0)\ne0,
 \qquad m(\T\setminus U)=0.
\]
Put $G=p\circ I$. The following assertions hold.

\begin{enumerate}
\item For every $\theta\in[0,2\pi)$ there are
$0<R_\theta\leq R_*$ and a curve
$\gamma_\theta\colon[0,R_\theta)\to\D$ such that
\begin{equation}\label{eq:blocking-lift}
 \gamma_\theta(0)=0,
 \qquad
 G(\gamma_\theta(t))=p(a)+t e^{i\theta}
 \quad(0\leq t<R_\theta).
\end{equation}
Moreover, $G'(\gamma_\theta(t))\ne0$ for $0\leq t<R_\theta$, the
curve $\gamma_\theta(t)$ leaves every compact subset of $\D$ as
$t\to R_\theta^-$, and
\[
 I(\gamma_\theta(t))\in K.
\]
Thus $\gamma_\theta$ gives the analytic continuation of the local inverse
of $G$ along this radial segment.

\item The function $I$ extends holomorphically across $U$, and
$|I|=1$ on $U$.
\end{enumerate}
\end{proposition}

\subsection{The inductive construction}

Choose a Jordan domain $V\Subset\D$ containing $a$ such that $p$ is
one-to-one on a neighborhood of $\overline V$. Put $z_*=p(a)$, choose
$\rho>0$ with
\[
 \overline{\Delta(z_*,\rho)}\subset p(V),
\]
and let
\[
 v\colon\Delta(z_*,\rho)\longrightarrow V
\]
be the inverse branch of $p$. Fix radii
\begin{equation}\label{eq:radii}
 0<r_1<r_2<\cdots<r_n<\cdots\longrightarrow r_*<\rho
\end{equation}
and set
\begin{equation}\label{eq:K-definition}
 K=v\bigl(\overline{\Delta(z_*,r_*)}\bigr)\subset V.
\end{equation}

For each $n$, divide the angular circle $\mathbb R/(2\pi\mathbb Z)$
at finitely many distinct points, with at least two points in each
division, and let
$\mathcal P_n$ be the family of complementary open angular intervals.
Choose these divisions so that each interval of $\mathcal P_{n+1}$ is
contained in an interval of $\mathcal P_n$. For $J\in\mathcal P_n$, put
\[
 Q_{n,J}
 =v\{z_*+r e^{i\theta}:r_n<r<r_{n+1},\ \theta\in J\}.
\]
We call $Q_{n,J}$ a corridor of generation $n$. Its threshold and exit
are the arcs at radii $r_n$ and $r_{n+1}$. Put
\[
 s_n=\frac{r_n+r_{n+1}}2,
 \qquad
 \Gamma_n=v\{z_*+s_n e^{i\theta}:0\leq\theta<2\pi\}.
\]
The curve $\Gamma_n$ is a Jordan curve in $V$. Its intersection with
$Q_{n,J}$ is the middle crosscut of the corridor.

For fixed $n$, the derivatives of $v$ and $v^{-1}$ are bounded on
the relevant compact annulus. Hence there are constants $C_n,d_n>0$
such that, when $J$ has angular length $\delta$, the threshold and exit
of $Q_{n,J}$ have diameter at most $C_n\delta$, while their distance
from $\Gamma_n$ is at least $d_n$. By \lemref{lem:small-gates}, with
$M=\Gamma_n$, the mesh of $\mathcal P_n$ can be chosen so small that
\begin{equation}\label{eq:planar-three-quarters}
 \omega(w,\T,\D\setminus(\overline\tau\cup\overline\xi))
 >\frac34
 \qquad(w\in\Gamma_n)
\end{equation}
for every corridor of generation $n$. The partitions are chosen
successively so that refinement and \eqref{eq:planar-three-quarters}
hold at every generation.

Let $A_{n,J}$ be the assembly associated with $Q_{n,J}$ and $\Gamma_n$
by \lemref{lem:assembly}, and put
\[
 \mathcal M_{n,J}=\pi_{n,J}^{-1}(\Gamma_n),
\]
where $\pi_{n,J}$ is its projection. For each assembly fix a conformal
map $\Psi_{n,J}\colon\D\to A_{n,J}$ and let $B_{n,J}\subset\T$ be the set in
that lemma. Equations \eqref{eq:planar-three-quarters} and
\eqref{eq:assembly-comparison} give
\begin{equation}\label{eq:assembly-three-quarters}
 \omega\bigl(\Psi_{n,J}^{-1}(x),B_{n,J},\D\bigr)>\frac34
 \qquad(x\in\mathcal M_{n,J}).
\end{equation}
Let
\[
 C_0=v\bigl(\Delta(z_*,r_1)\bigr)
\]
be the central sheet. Attach the threshold of every first-generation
assembly to the corresponding open arc of $\partial C_0$. Inductively,
subdivide the exit of an assembly of generation $n$ according to
$\mathcal P_{n+1}$, and attach the threshold of the appropriate child
assembly to each open subarc. All endpoints of attaching arcs are omitted
from every incident piece, including the subdivision points on the parent
exits. Denote the resulting surface by $R$, and let
\[
 \pi\colon R\longrightarrow\D
\]
be the projection induced by the projections of the pieces.

\begin{lemma}\label{lem:tree-gluing}
The space $R$ is a Hausdorff, second-countable Riemann surface, $\pi$
is locally biholomorphic, and $R$ is simply connected. Every compact
subset of $R$ meets only finitely many generations.
\end{lemma}

\begin{proof}
The incidence graph of the pieces is a countable locally finite rooted
tree. Adjoin the attaching arcs to each piece before taking the quotient.
These arcs are closed and pairwise disjoint in each piece because their
endpoints are omitted. Only finitely many occur in each piece. Thus the
quotient map from the disjoint union of the pieces is closed and has
finite fibres, so the quotient is Hausdorff. The collars on opposite
sides of each attaching arc give a holomorphic chart through the
projection. Each piece and each attaching arc has a countable cover by
such charts. Since there are countably many pieces and attaching arcs,
$R$ is second-countable. Thus $R$ is a Riemann surface and $\pi$ is
locally biholomorphic.

To keep track of the generations, set $\ell=0$ on the central sheet. On
each assembly of generation $n$ choose a continuous function with values in
$[n-1,n]$, equal to $n-1$ on a collar of the threshold and to $n$ on
a collar of the exit. Such a function exists because the two closed
collars are disjoint. These functions agree on every attaching arc and
therefore define a continuous function $\ell\colon R\to[0,\infty)$.
Since $\ell$ is bounded on every compact subset of $R$, such a subset
meets only finitely many generations. Since every generation contains
finitely many assemblies, it meets only finitely many pieces.

A finite connected union of pieces is simply connected. Indeed, add the
pieces one at a time and apply van Kampen's theorem. Every new piece is
simply connected and is attached along a contractible collar. Every loop
in $R$ has compact image, hence lies in such a finite union. It is
therefore null-homotopic in $R$.
\end{proof}

The surface $R$ is noncompact, and $\pi$ is bounded and nonconstant.
So uniformization and Liouville's theorem give a conformal map
\begin{equation}\label{eq:uniformization}
 \Phi\colon\D\longrightarrow R.
\end{equation}
Let $x_0\in C_0$ be the point with $\pi(x_0)=a$, normalize
$\Phi(0)=x_0$, and define
\begin{equation}\label{eq:I-definition}
 I=\pi\circ\Phi.
\end{equation}
Both $\pi$ and $\Phi$ are locally biholomorphic. Thus $I$ is
locally univalent, $I(0)=a$, and $I'(0)\ne0$.

For $n\geq1$, let
\[
 \calM_n=\bigcup_{J\in\mathcal P_n}\mathcal M_{n,J}.
\]
Regard the central sheet as generation zero. The separation property in
\lemref{lem:assembly} implies that $\calM_n$ separates the generations
at most $n-1$ from the generations at least $n+1$. Since
$\mathcal M_{n,J}$ is the full inverse image of $\Gamma_n$ in its
assembly, this separation also holds for paths entering a wing.

Let $\calS\subset\D$ be the set of projections of all
omitted endpoints. Note that this set is countable.

\begin{lemma}\label{lem:finite-level}
Let $\alpha\colon[0,1)\to R$ be a continuous curve that leaves every
compact subset of $R$ as $t\to1$. Suppose that $\alpha$ meets only
finitely many generations and that
\[
 \pi(\alpha(t))\longrightarrow\lambda\in\D.
\]
Then $\lambda\in\calS$.
\end{lemma}

\begin{proof}
Choose $N$ such that $\alpha$ meets no generation greater than $N$.
Let $Y$ be the union of the central sheet and all assemblies through
generation $N$, together with their attaching arcs. Let $S_Y$ be the set
of projections of their corners and
all subdivision endpoints on their attaching arcs, including those used
to attach generation $N+1$. This is a finite subset of $\calS$.

Suppose that $\lambda\notin\calS$. Choose concentric discs
$B'\Subset B\Subset\D\setminus S_Y$ about $\lambda$ so small that $B$
meets at most one side of each corridor, and the projection of every cut
bank or attaching arc meets $B$, if at all, in a single analytic arc.
For all large $t$, $\pi(\alpha(t))\in B'$, and the tail of $\alpha$ lies
in one component $Y_0$ of $Y\cap\pi^{-1}(B)$.

Over $B$, each covering piece is a disjoint union of one-sheeted
components. The corridor pieces are already one-sheeted. Each cut bank
or attaching arc over $B$ joins at most one such component from each side.
The finite tree of attachments and the choice of $B$ show that $Y_0$
contains at most one such component from each piece. Any boundary arc
of $Y_0$ that projects into $B$ is an analytic attaching arc of $R$. Adjoining these
arcs, or equivalently adjoining the corresponding collars in the finitely many
adjacent child assemblies, completes the local coordinates across them.
Since $\overline{B'}\subset B$, the set
$Y_0\cap\pi^{-1}(\overline{B'})$ has compact closure in $R$.
This closure contains the tail of $\alpha$, contrary to the assumption
that $\alpha$ leaves every compact subset of $R$.
\end{proof}

\subsection{Harmonic measure}

We regard Brownian motion on $R$ through the uniformization
\eqref{eq:uniformization}. Conformal maps preserve Brownian paths up to a
change of time and therefore preserve hitting probabilities. For
$n\geq1$, set
\[
 q_n=\sup_{x\in\calM_n}
 \mathbb P_x\{\text{the path reaches }\calM_{n+1}\}.
\]
Start at a point of $\calM_n$, and stop the path when, in its current
assembly, it first reaches the threshold, the exit, or the boundary
represented by $B_{n,J}$. The remaining boundary set has harmonic
measure zero by \lemref{lem:assembly}. From
\eqref{eq:assembly-three-quarters}, the total probability of reaching the
threshold or exit before $B_{n,J}$ is less than $1/4$. In particular,
each of those two events has probability at most $1/4$.

If the exit is reached first, the conditional probability of subsequently
hitting $\calM_{n+1}$ is at most one. If the threshold is reached first,
then every later path to $\calM_{n+1}$ must first return to $\calM_n$.
The strong Markov property gives
\begin{equation}\label{eq:qn-estimate}
 q_n\leq\frac14+\frac14q_n,
 \qquad q_n\leq\frac13.
\end{equation}
By applying the strong Markov property at the successive first hitting
times of $\calM_1,\ldots,\calM_N$, we obtain
\begin{equation}\label{eq:generation-probability}
 \mathbb P_{x_0}\{\text{the path reaches }\calM_N\}
 \leq 3^{1-N}.
\end{equation}
A continuous path from $x_0$ that meets infinitely many generations must
meet every $\calM_N$. Equation \eqref{eq:generation-probability}
therefore gives
\begin{equation}\label{eq:unbounded-level-null}
 \mathbb P_{x_0}\{\text{the path meets infinitely many generations}\}=0.
\end{equation}

We next prove that $I$ is inner. Let $Z_t$ be Brownian motion in
$\D$, started at $0$ and stopped at its first exit time $\tau$.
The bounded analytic function $I$ has radial limits $I^*$ almost
everywhere. By the Poisson integral formula and the martingale convergence
theorem, $I(Z_t)\to I^*(Z_\tau)$ almost surely as $t\to\tau$. See
\cite[Corollary~I.2.5 and Appendix~F]{GarnettMarshall2005}.
Put $X_t=\Phi(Z_t)$. Then $X_t$ leaves every compact subset of $R$ as
$t\to\tau$.

Suppose that the limit of $\pi(X_t)=I(Z_t)$ lies in $\D$. Outside
the null event in \eqref{eq:unbounded-level-null}, the path meets only
finitely many generations, and \lemref{lem:finite-level} shows that the
limit belongs to $\calS$. For
each $\lambda\in\calS$, Privalov's uniqueness theorem gives
\begin{equation}\label{eq:fixed-boundary-value-null}
 m\{\zeta\in\T:I^*(\zeta)=\lambda\}=0,
\end{equation}
since $I-\lambda$ is a nonzero bounded analytic function. As $\calS$
is countable, the probability that the boundary limit of $I$ lies in
$\D$ is zero. Therefore
\[
 |I^*(\zeta)|=1
 \quad\text{for almost every }\zeta\in\T,
\]
and $I$ is inner.

\subsection{Analytic continuation across boundary arcs}

Every corridor and every attaching arc projects into the compact set $K$
in \eqref{eq:K-definition}. Choose a compact set $K_1$ with
$K\Subset K_1\Subset\D$. Let $\zeta\in\T$ be a point at which
$I^*(\zeta)=\lambda\in\T$. Since
$I(r\zeta)\to\lambda$, the curve $\Phi(r\zeta)$ eventually projects
outside $K_1$. It can then cross neither an attaching arc nor a bank of
a cut, and consequently remains in one covering surface.

Choose an open arc $J$ of $\T$ containing $\lambda$ and a simply
connected Jordan domain $C_J\subset\D$ whose boundary contains $J$ and
whose closure in $\overline\D$ is disjoint from $K_1$.
For $r$ sufficiently close to one, $\Phi(r\zeta)$ lies in a component
$\widetilde C$ of the inverse image of $C_J$ in that covering surface.
By \lemref{lem:wing},
\[
 \pi\colon\widetilde C\longrightarrow C_J
\]
is biholomorphic and $\widetilde C$ has an analytic boundary arc over
$J$.

Put
\[
 g=\Phi^{-1}\circ(\pi|_{\widetilde C})^{-1}.
\]
The boundary argument in the proof of \lemref{lem:assembly} extends $g$
conformally across $J$. Since $I\circ g=\mathrm{id}$, its inverse extends
$I$ holomorphically across the open arc
$U_{\widetilde C}=g(J)\subset\T$, with $|I|=1$ there.

Choose the arcs $J$ from a countable basis of $\T$. There are
countably many covering surfaces, and every covering has countably many
components over a fixed collar. Hence the family of arcs
$U_{\widetilde C}$ obtained in this way is countable. Let $U$ be their
union. The set $U$ is open. Since $I$ is inner, $I^*(\zeta)\in\T$ for
almost every $\zeta\in\T$. For such a point $\zeta$, the preceding
argument fixes a covering surface containing the tail of
$\Phi(r\zeta)$. Choosing a basis arc $J$ around $I^*(\zeta)$ then yields
a component $\widetilde C$ for which $\zeta\in U_{\widetilde C}$.
Thus
\begin{equation}\label{eq:U-full-measure}
 m(\T\setminus U)=0.
\end{equation}

\subsection{Radial lifts}

Let $\Theta\subset[0,2\pi)$ be the countable set of representatives of
the endpoints of all the angular partitions. If $\theta\notin\Theta$,
there is a unique nested sequence of partition intervals
$J_n\in\mathcal P_n$ containing the class of $\theta$. The
lift from $x_0$ of
\begin{equation}\label{eq:radial-base-curve}
 t\longmapsto v(z_*+t e^{i\theta}),
 \qquad 0\leq t<r_*,
\end{equation}
passes through the corridor belonging to $J_n$ at every generation. By
\lemref{lem:tree-gluing}, it leaves every compact subset of $R$ as
$t\to r_*^-$.

If $\theta\in\Theta$, let $n$ be the first generation in which
$\theta$ is a partition endpoint. The same lift is defined for
$0\leq t<r_n$ and approaches, as $t\to r_n^-$, an omitted endpoint of
the attaching arcs. The local charts used in the gluing show that the
lift has no accumulation point in $R$, and hence leaves every compact
subset of $R$. Define
\[
 R_\theta=
 \begin{cases}
 r_*,&\theta\notin\Theta,\\
 r_n,&\theta\in\Theta\text{ first occurs at generation }n.
 \end{cases}
\]
Let $\widetilde\gamma_\theta$ be the lifted curve in $R$, and put
\[
 \gamma_\theta=\Phi^{-1}\circ\widetilde\gamma_\theta.
\]
Then \eqref{eq:blocking-lift} holds with $R_*=r_*$, and
\[
 I(\gamma_\theta(t))=v(z_*+t e^{i\theta})\in K.
\]
Since $p$ is one-to-one on $V$, $p'$ has no zeros there. Together
with the local univalence of $I$, this gives
\[
 G'(\gamma_\theta(t))
 =p'(I(\gamma_\theta(t)))I'(\gamma_\theta(t))\ne0.
\]
The curves $\gamma_\theta$ therefore give the asserted analytic
continuations of the local inverse of $G$. This completes the proof of
\propref{prop:blocking}.

\section{A holomorphic function with positive real part}\label{sec:barrier}

Let $A(\D)$ denote the disc algebra, which consists of the functions
that are holomorphic in $\D$ and continuous on $\overline\D$. By Fatou's
theorem, every closed subset of $\T$ of arclength measure zero is a peak set for
$A(\D)$ \cite[p.~393]{Fatou1906}. See also \cite[Section~3]{Noell2020}.

\begin{lemma}\label{lem:barrier}
Let $I\colon\D\to\D$ be holomorphic. Suppose that there is an open set
$U\subseteq\T$ with $m(\T\setminus U)=0$ such that $I$ extends
holomorphically across $U$ and $|I|=1$ on $U$. Then there is a function
$H\in\mathcal O(\D)$ that satisfies $\Real{H}>1$ and
\begin{equation}\label{eq:barrier-property}
 \Real{H(\zeta_j)}\longrightarrow+\infty
\end{equation}
whenever $|\zeta_j|\to1$ and the sequence $(I(\zeta_j))$ remains in a
compact subset of $\D$.
\end{lemma}

\begin{proof}
Put
\[
 E=\T\setminus U.
\]
Then $E$ is closed and $m(E)=0$. If $E=\varnothing$, continuity on
$\overline\D$ and $|I|=1$ on $\T$ exclude the sequences in the assertion,
and we may take $H\equiv2$. Suppose that $E\ne\varnothing$. Fatou's
theorem gives a function $\chi\in A(\D)$ that satisfies
\[
 \chi=1\quad\text{on }E,
 \qquad |\chi|<1\quad\text{on }\overline\D\setminus E.
\]
Since $\Real{1-\chi}>0$ in $\D$, we may use the principal branch of the
square root and define
\begin{equation}\label{eq:single-peak-barrier}
 H=1+(1-\chi)^{-1/2}.
\end{equation}
Then
\[
 \left|\arg (1-\chi)^{-1/2}\right|<\frac\pi4,
 \qquad \Real{H}>1.
\]

Suppose that $|\zeta_j|\to1$ and that $I(\zeta_j)$ remains in a compact
set $L\subset\D$. We claim that
\begin{equation}\label{eq:distance-to-E}
 \dist(\zeta_j,E)\longrightarrow0.
\end{equation}
Otherwise, a subsequence would converge to a point
$\xi\in\T\setminus E=U$. Holomorphic extension across $U$ would give,
along this subsequence,
\[
 I(\zeta_j)\longrightarrow I(\xi),
 \qquad |I(\xi)|=1,
\]
contrary to $I(\zeta_j)\in L\subset\D$.

Uniform continuity of $\chi$ on $\overline\D$, together with
\eqref{eq:distance-to-E} and $\chi=1$ on $E$, gives
$\chi(\zeta_j)\to1$. The preceding bound on the argument therefore gives
$\Real{H(\zeta_j)}\to+\infty$.
\end{proof}

\section{Construction of a proper holomorphic immersion}\label{sec:completion}

\begin{lemma}\label{lem:explicit-proper-disc}
Let
\[
 \Omega=\C\setminus\{0,1\},
\]
and let $p\colon\D\to\Omega$ be a universal covering. Put
\begin{equation}\label{eq:explicit-second-coordinate}
 s(\zeta)=\frac{p'(\zeta)}{p(\zeta)(1-p(\zeta))}.
\end{equation}
Then
\begin{equation}\label{eq:explicit-proper-disc}
 P=(p,s)\colon\D\longrightarrow\C^2,
 \qquad P(\D)\subset(\C^*)^2,
\end{equation}
is proper. Both coordinate functions are zero-free, and
$p'(\zeta)\ne0$ for every $\zeta\in\D$.
\end{lemma}

\begin{proof}
By the uniformization theorem, the universal cover of $\Omega$ is either
$\D$ or $\C$. The latter possibility is excluded by the little Picard
theorem, so a universal covering $p\colon\D\to\Omega$ exists. Since $p$
is locally biholomorphic, $p'$ has no zeros. The function $p$ omits $0$
and $1$, so $s$ is holomorphic and zero-free.

Let $\lambda_X$ denote the Poincar\'e density of curvature $-1$ on a
hyperbolic plane domain $X$ \cite{BeardonMinda2007}. Covering
invariance gives
\begin{equation}\label{eq:covering-metric-identity}
 \lambda_{\D}(\zeta)
 =\lambda_{\Omega}(p(\zeta))|p'(\zeta)|.
\end{equation}
For every $M>0$, we have
\begin{equation}\label{eq:CM-definition}
 C_M=
 \sup_{\substack{w\in\Omega\\ |w|\leq M}}
 \lambda_{\Omega}(w)|w|\,|1-w|<\infty.
\end{equation}
To prove this, choose $0<\delta<1/4$. The expression is bounded on the
compact set where $|w|\leq M$, $|w|\geq\delta$, and
$|1-w|\geq\delta$. If $0<|w|<\delta$, monotonicity and the inclusion
\[
 \{0<|z|<2\delta\}\subset\Omega
\]
give
\[
 \lambda_{\Omega}(w)
 \leq\lambda_{\{0<|z|<2\delta\}}(w)
 =\frac{1}{|w|\log(2\delta/|w|)}.
\]
The last identity follows from the exponential covering of the punctured
disc by a half-plane. Hence
\[
 \lambda_{\Omega}(w)|w|\,|1-w|
 \leq\frac{|1-w|}{\log2}
 \qquad(0<|w|<\delta).
\]
The same argument after translation by $1$ gives a uniform bound when
$0<|1-w|<\delta$. This proves \eqref{eq:CM-definition}.

Suppose that $|\zeta_j|\to1$ while $P(\zeta_j)$ remains bounded. Choose
$M,N>0$ such that
\[
 |p(\zeta_j)|\leq M,
 \qquad |s(\zeta_j)|\leq N.
\]
By \eqref{eq:explicit-second-coordinate},
\eqref{eq:covering-metric-identity}, and \eqref{eq:CM-definition},
\[
 \lambda_{\D}(\zeta_j)
 =\lambda_{\Omega}(p(\zeta_j))
  |p(\zeta_j)|\,|1-p(\zeta_j)|\,|s(\zeta_j)|
 \leq C_M N.
\]
This contradicts
\[
 \lambda_{\D}(\zeta)=\frac{2}{1-|\zeta|^2}
 \longrightarrow+\infty
 \qquad(|\zeta|\to1).
\]
Thus $P$ is proper.
\end{proof}

Now fix a universal covering $p\colon\D\to\Omega$ as in
\lemref{lem:explicit-proper-disc}, and let $s$ be defined by
\eqref{eq:explicit-second-coordinate}. Apply \propref{prop:blocking} with
this function $p$ and $a=0$. We obtain a locally univalent inner function
$I\colon\D\to\D$ with $I(0)=0$, a compact set $K\subset\D$, and an
open set $U\subset\T$ that satisfy the conclusions of that proposition.
By \lemref{lem:barrier}, there is a function $H\in\mathcal O(\D)$
that satisfies $\Real{H}>1$ and \eqref{eq:barrier-property}.

Put
\begin{equation}\label{eq:G-and-Q}
 G=p\circ I,
 \qquad Q=(s\circ I)H.
\end{equation}
For $\eps\in\C$ with $|\eps|<1$, define
\begin{equation}\label{eq:J-epsilon}
 J_{\eps}(\zeta)
 =\left(G(\zeta),Q(\zeta)(1+\eps\zeta)\right).
\end{equation}

\begin{proposition}\label{prop:J-proper}
For every $\eps\in\C$ with $|\eps|<1$, the map
\[
 J_{\eps}\colon\D\longrightarrow\C^2
\]
is a proper holomorphic immersion, and
$J_{\eps}(\D)\subset(\C^*)^2$.
\end{proposition}

\begin{proof}
The functions $p$, $s$, and $H$ are zero-free, and
\[
 |1+\eps\zeta|\geq1-|\eps|>0
 \qquad(\zeta\in\D).
\]
Hence both coordinates of $J_{\eps}$ are zero-free. Since $p'$ and $I'$
have no zeros, we also have
\begin{equation}\label{eq:G-nowhere-critical}
 G'(\zeta)=p'(I(\zeta))I'(\zeta)\ne0
 \qquad(\zeta\in\D).
\end{equation}
Thus $J_{\eps}$ is an immersion.

To prove properness, suppose that $|\zeta_j|\to1$ while
$J_{\eps}(\zeta_j)$ remains bounded. Since
$|1+\eps\zeta_j|\geq1-|\eps|$, both $G(\zeta_j)$ and $Q(\zeta_j)$
remain bounded. The inequality $|H|\geq\Real{H}>1$ gives
\[
 |s(I(\zeta_j))|
 =\frac{|Q(\zeta_j)|}{|H(\zeta_j)|}
 \leq |Q(\zeta_j)|.
\]
Therefore
\[
 P(I(\zeta_j))
 =\bigl(p(I(\zeta_j)),s(I(\zeta_j))\bigr)
\]
remains bounded. Since $P$ is proper, the points $I(\zeta_j)$ lie in a
compact subset $L$ of $\D$. By \lemref{lem:barrier},
\[
 \Real{H(\zeta_j)}\longrightarrow+\infty.
\]
The zero-free function $s$ is bounded away from zero on $L$, and
$|1+\eps\zeta_j|\geq1-|\eps|$. Thus the modulus of the second
coordinate of $J_{\eps}(\zeta_j)$ tends to infinity, a contradiction.
\end{proof}

We now choose $\eps$ so that $J_{\eps}(0)$ has a unique preimage. By
\eqref{eq:G-nowhere-critical}, the fiber $G^{-1}(G(0))$ is discrete and
countable. For $\eta\ne0$ in this fiber, the equality
$J_{\eps}(\eta)=J_{\eps}(0)$ holds precisely when
\[
 Q(\eta)(1+\eps\eta)=Q(0),
\]
or equivalently, when
\begin{equation}\label{eq:epsilon-exceptional}
 \eps=\frac{Q(0)/Q(\eta)-1}{\eta}.
\end{equation}
There is at most one exceptional value for each such $\eta$. Choose
$\eps\in\D$ outside this countable set, and put
\begin{equation}\label{eq:J-and-basepoint}
 J=J_{\eps},
 \qquad (z_0,w_0)=J(0).
\end{equation}
Then $J\colon\D\to\C^2$ is a proper holomorphic immersion,
$J(\D)\subset(\C^*)^2$, and
\begin{equation}\label{eq:unique-base-fiber}
 J^{-1}(z_0,w_0)=\{0\}.
\end{equation}

\section{Proof of Theorem~\ref{thm:main}}\label{sec:main-proof}

\begin{proof}[Proof of Theorem~\ref{thm:main}]
Let
\[
 A=J(\D)\subset\C^2.
\]
By \propref{prop:J-proper} and Remmert's proper mapping theorem
\cite{Remmert1956}, $A$ is a closed analytic subset of $\C^2$. It is
irreducible because $\D$ is irreducible, and it has pure dimension one
because $J$ is a nonconstant immersion. Thus $A$, with its reduced
structure, is an effective Cartier divisor on $\C^2$.

The exponential sequence gives the exact segment
\[
 H^1(\C^2,\mathcal O)
 \longrightarrow H^1(\C^2,\mathcal O^*)
 \longrightarrow H^2(\C^2,\Z).
\]
By Cartan's Theorem~B, $H^1(\C^2,\mathcal O)=0$. Since $\C^2$ is
contractible, $H^2(\C^2,\Z)=0$. It follows that
\[
 H^1(\C^2,\mathcal O^*)=0,
\]
so every Cartier divisor on $\C^2$ is principal. Since the divisor
defined by $A$ is effective and reduced, there is a reduced entire
function $F\in\mathcal O(\C^2)$ that satisfies
\begin{equation}\label{eq:A-zero-set}
 A=\{(z,w)\in\C^2:F(z,w)=0\}.
\end{equation}

We next show that $F_w(z_0,w_0)\ne0$.
Since $J$ is an immersion, there is a disc $D_0\Subset\D$ about $0$ on
which $J$ is an embedding. A sufficiently small neighborhood $W$ of
$(z_0,w_0)$ satisfies $J^{-1}(W)\subset D_0$. Otherwise, there would be
points $\eta_j\in\D\setminus D_0$ with $J(\eta_j)\to(z_0,w_0)$.
Properness of $J$ would give a convergent subsequence
$\eta_j\to\eta\in\D\setminus D_0$. Continuity would then imply
$J(\eta)=(z_0,w_0)$, contrary to \eqref{eq:unique-base-fiber}. Hence $A$
agrees near $(z_0,w_0)$ with the embedded image of $D_0$ and is smooth
there.

The first component of $J$ is $G=p\circ I$, and
\begin{equation}\label{eq:G-prime-base}
 G'(0)=p'(0)I'(0)\ne0.
\end{equation}
Thus the tangent line to $A$ at $(z_0,w_0)$ has a nonzero $z$-component.
Since $F$ is reduced and $A$ is smooth at $(z_0,w_0)$, we have
$dF(z_0,w_0)\ne0$. The tangent line is the kernel of $dF=(F_z,F_w)$.
Consequently,
\begin{equation}\label{eq:Fw-nonzero}
 F_w(z_0,w_0)\ne0.
\end{equation}

Let $\varphi$ be the implicit germ at $z_0$ with $\varphi(z_0)=w_0$.
Since $G'(0)\ne0$, there is a local inverse $\sigma$ of $G$ with
$\sigma(z_0)=0$. The parametrization $J$ gives
\begin{equation}\label{eq:implicit-parametrization}
 \varphi(z)=Q(\sigma(z))\bigl(1+\eps\sigma(z)\bigr)
\end{equation}
near $z_0$.

Fix $\theta\in[0,2\pi)$. By \propref{prop:blocking}, there are
$0<R_\theta\leq R_*$ and a curve $\gamma_\theta$ that satisfies
\[
 G(\gamma_\theta(t))=z_0+t e^{i\theta},
 \qquad 0\leq t<R_\theta,
\]
with $G'(\gamma_\theta(t))\ne0$. The local inverse theorem and uniqueness
of analytic continuation show that
\[
 \sigma(z_0+t e^{i\theta})=\gamma_\theta(t)
\]
is the analytic continuation of the local inverse along the ray. As
$t\to R_\theta^-$, the point $\gamma_\theta(t)$ leaves every compact
subset of $\D$, while
\[
 I(\gamma_\theta(t))\in K\subset\D.
\]
Since $K$ is compact, \lemref{lem:barrier} gives
\begin{equation}\label{eq:H-along-ray}
 \Real{H(\gamma_\theta(t))}\longrightarrow+\infty.
\end{equation}

Put
\[
 c_K=\min_{\eta\in K}|s(\eta)|>0.
\]
By \eqref{eq:implicit-parametrization}, \eqref{eq:G-and-Q}, and
\eqref{eq:H-along-ray},
\begin{equation}\label{eq:implicit-blowup}
 |\varphi(z_0+t e^{i\theta})|
 \geq c_K(1-|\eps|)\Real{H(\gamma_\theta(t))}
 \longrightarrow+\infty
 \qquad(t\to R_\theta^-).
\end{equation}
A holomorphic continuation through $z_0+R_\theta e^{i\theta}$ would be
locally bounded there, contrary to \eqref{eq:implicit-blowup}. Since
$\theta$ was arbitrary, the theorem follows.
\end{proof}

\end{document}